\documentclass[preprint]{imsart} 
\usepackage{amsmath, latexsym, amsthm, amsfonts,bm,amssymb} 
\usepackage[dvipsnames]{xcolor}
\usepackage{appendix,todonotes}

\usepackage{natbib} 
\usepackage{url} 
\usepackage{enumerate}
\usepackage{mathtools} 

\usepackage{natbib} 
\bibpunct{(}{)}{;}{a}{}{,} 

\theoremstyle{plain}
\newtheorem{thm}{Theorem}
\newtheorem{prop}[thm]{Proposition}
\newtheorem{lemma}[thm]{Lemma}

\usepackage[normalem]{ulem}

\theoremstyle{remark}

\newcommand{\bcdot}{\boldsymbol{\cdot}}

\newcommand{\half}{\frac{1}{2}}

\newcommand{\pp}{\mathbb{P}}
\renewcommand{\qq}{\mathbb{Q}}
\newcommand{\ee}{\mathbb{E}}
\newcommand{\rr}{\mathbb{R}}
\newcommand{\cf}{\mathcal{F}}

\newcommand{\dd}{\mathrm{d}}
\newcommand{\Pp}{\mathbb{P}}

\newcommand{\fpspace}{(\Omega,\mathcal{F},\mathbb{F},\Pp)}

\newcommand{\one}{\mathbf{1}}

\allowdisplaybreaks

\usepackage{tikz,hyperref}

\definecolor{lime}{HTML}{A6CE39}
\DeclareRobustCommand{\orcidicon}{%
	\begin{tikzpicture}
	\draw[lime, fill=lime] (0,0) 
	circle [radius=0.16] 
	node[white] {{\fontfamily{qag}\selectfont \tiny ID}};
	\draw[white, fill=white] (-0.0625,0.095) 
	circle [radius=0.007];
	\end{tikzpicture}
	\hspace{-2mm}
}

\foreach \x in {A, ..., Z}{%
	\expandafter\xdef\csname orcid\x\endcsname{\noexpand\href{https://orcid.org/\csname orcidauthor\x\endcsname}{\noexpand\orcidicon}}
}

\begin{document}

\begin{frontmatter}
\title{Weak solutions to gamma-driven stochastic differential equations}
\runtitle{Weak solutions for gamma-driven SDEs}
\runauthor{Belomestny, Gugushvili, Schauer, Spreij}

\begin{aug}
\author[A]{\fnms{Denis} \snm{Belomestny}\ead[label=e1]{denis.belomestny@uni-due.de}},
\author[B]{\fnms{Shota} \snm{Gugushvili}\ead[label=e2]{shota@yesdatasolutions.com}}
\and
\author[C]{\fnms{Moritz} \snm{Schauer}\ead[label=e3]{smoritz@chalmers.se}}
\and
\author[D]{\fnms{Peter} \snm{Spreij}\ead[label=e4]{spreij@uva.nl}}
\address[A]{Department,
Faculty of Mathematics\\
Duisburg-Essen University\\
Thea-Leymann-Str.~9\\
D-45127 Essen\\
Germany\\
and Faculty of Computer Sciences\\
HSE University\\
Moscow, Russian Federation
\printead{e1}}

\address[B]{Biometris\\
	Wageningen University \& Research\\
	Postbus 16\\
	6700 AA Wageningen\\
	The Netherlands
\printead{e2}}

\address[C]{Department of Mathematical Sciences\\
	Chalmers University of Technology and University of Gothenburg\\
	SE-412 96 G\"{o}teborg\\
	Sweden
\printead{e3}}

\address[D]{Korteweg-de Vries Institute for Mathematics\\
Universiteit van Amsterdam\\
P.O.\ Box 94248\\
1090 GE Amsterdam\\
The Netherlands\\
and Institute for Mathematics, Astrophysics and Particle Physics\\
Radboud University\\
Nijmegen\\
Netherlands
\printead{e4}}

\end{aug}

\begin{abstract}
We study a stochastic differential equation driven by a gamma process, for which we give results on the existence of weak solutions under conditions on the volatility function. To that end we provide results on the density process between the laws of solutions with different volatility functions.
\end{abstract}


\end{frontmatter}

\section{Introduction}

The goal of the present paper is to give conditions such that the L\'evy-driven stochastic differential equation 
\begin{equation} \label{eq:sde}
\dd X_t=\sigma(X_{t-})\,\dd L_t,\, X_0= 0 
\end{equation}
has a weak solution that is unique in law.
Here
$L$ is a gamma process with $L_0=0$ and therefore $L$ is a subordinator, i.e.\ a stochastic process with monotonous sample paths. Furthermore, $L$ has
a L\'evy measure $\nu$ admitting the L\'evy density
\begin{equation}\label{eq:v}
v(x)= \alpha x^{-1}\exp(-\beta x),\, x>0,
\end{equation}
where $\alpha$ and $\beta$ are two positive constants. The process $L$ has independent increments,  and $L_t-L_s$ has a $\operatorname{Gamma}(\alpha (t-s),\beta)$ distribution for $t>s$. 
Recall that the $\operatorname{Gamma}(a,b)$ distribution has a density given by
$
x \mapsto \frac{b^a}{\Gamma(a)} x^{a-1}e^{-b x}$ for $x>0$.

Under the assumption that the function $\sigma$ (in view of financial applications we refer to it as volatility function) is measurable, positive and satisfies a linear growth condition, we will see in Theorem~\ref{thm:weak} that Equation~\eqref{eq:sde} admits a weak solution that is unique in law. This is the main result of the present paper.
Note that under the stronger condition that $\sigma$ is 
Lipschitz continuous, Equation~\eqref{eq:sde} even has a unique strong solution, see~\cite[Theorem~V.6]{Protter2004}. 


We will briefly outline the relevance of gamma processes. They form a special class of L\'evy processes (see, e.g., \cite{kyprianou14}), are a fundamental modelling tool in several fields, e.g. reliability (see \cite{noortwijk09}) and risk theory (see \cite{dufresne91}). Since the driving gamma process $L$ in \eqref{eq:sde} has non-decreasing sample paths and the volatility function $\sigma$ is non-negative, also the process $X$ has non-decreasing sample paths. Such processes find applications across various fields. A reliability model as in \eqref{eq:sde} has been thoroughly investigated from a probabilistic point of view in \cite{wenocour89}, and constitutes a far-reaching generalisation of a basic gamma model. Furthermore, non-decreasing processes are ideally suited to model revenues from an innovation: in \cite{chance08}, the authors study the question of pricing options on movie box office revenues that are modelled through a gamma-like stochastic process. Another potential application is in modelling the evolution of forest fire sizes over time, as in \cite{reed02}.

Any practical application of the model \eqref{eq:sde} would require knowledge of the volatility function $\sigma$, that has to be inferred from observations on the process $X$. This is  a statistical problem to which we present a nonparametric Bayesian approach in \cite{bgsslevy}. The obtained results in this paper assume either a piecewise constant volatility function $\sigma$ or a H\"older continuous one. In both cases one needs existence of weak solutions to \eqref{eq:sde} and a likelihood ratio. The present paper covers those two cases and provides the foundations of the statistical analysis.
For a survey of other contributions to statistical inference for L\'evy driven SDEs we also refer to \cite{bgsslevy}.

\section{Absolute continuity and likelihood}
\label{sec:likelihood}

In the proof our main result, Theorem~\ref{thm:weak}, we need the likelihood ratio between different laws of solutions to \eqref{eq:sde}. In this section we give the relevant results.

Let $\fpspace$ be a filtered probability space and let $(L_t)_{t\geq 0}$ be a gamma process adapted to $\mathbb{F}$, whose L\'evy measure admits the density $v$ given by \eqref{eq:v}.
Assume that $X$ is a (weak) solution to \eqref{eq:sde},
and assume that $X$ is observed on an interval $[0,T]$. We denote by $\pp^\sigma_T$ (a probability measure on $\cf^X_T=\sigma(X_t, 0\leq t\leq T)$) its law. In agreement with this notation we let $\pp^1_T$ be the law of $X$ when $\sigma\equiv1$, in which case $X_t=L_t,\, t\in [0,T]$. The measure $\pp^1_T$ will serve as a reference measure. The choice $\sigma=1$ for obtaining a reference measure is natural, but also arbitrary. Many other choices for the function $\sigma$ are conceivable, in particular other constant functions.
The question we are going to investigate first is under which conditions the laws $\pp^\sigma_T$ and $\pp^1_T$ are equivalent. 
Suppose that the process $\sigma(X_{t-}),$ $t\in [0,T]$ is strictly positive and define
\begin{equation}\label{eq:vt}
v^\sigma(t,x)=\frac{1}{\sigma(X_{t-})}v\left(\frac{x}{\sigma(X_{t-})}\right).
\end{equation}
First we show that for $X$ given by \eqref{eq:sde}, its compensated jump measure under $\pp^\sigma_T$ is determined by \eqref{eq:vt}.
\begin{lemma}\label{lemma:vsigma}
Assume that \eqref{eq:sde} admits a weak solution for a given measurable function $\sigma$ with $\sigma(X_{s-})>0$ a.s. for all $s\geq 0.$
Under the measure $\pp^\sigma_T$, the third characteristic of the semimartingale $X$, its compensated jump measure $\nu^\sigma(\dd x,\dd t)$, is given by $\nu^\sigma(\dd x,\dd t)=v^\sigma(t,x)\, \dd x\dd t$.
\end{lemma}

\begin{proof}
Let $f$ be a bounded measurable function and let $X$ be given by \eqref{eq:sde}.
Then 
(the summation is only for those $s$ with $\Delta X_s>0$, and $t\in [0,T]$ is arbitrary) with $\mu^L$ being the jump measure of $L$,
\[
\sum_{s\leq t}f(\Delta X_s)
=\sum_{s\leq t}f(\sigma(X_{s-})\Delta L_s)=\int_0^t\int_{(0,\infty)}f(\sigma(X_{s-})z)\mu^L(\dd z,\dd s),
\]
which is the sum of a local martingale $M$ under $\pp$, adapted to $\mathbb{F}$, and the predictable process $\int_0^t\int_{(0,\infty)}f(\sigma(X_{s-})z)v(z)\,\dd z\dd s$. As  the latter expression  only depends on the process $X$, the local martingale $M$ is also adapted to $\mathbb{F}^X=\{\cf^X_t, t\geq 0\}$.
By a simple change of variable, the double integral equals to  
\[
\int_0^t \int_{(0,\infty)}f(x)(1/\sigma(X_{s-}))v\bigl(x/\sigma(X_{s-})\bigr)\,\dd x\dd s.
\] 
Since $\pp^\sigma_T$ is the law of $X$ for $X$ given by \eqref{eq:sde}, 
the expression in the display, again depending on $X$ only, is also  the $\mathbb{F}^X$-compensator under $\pp^\sigma_T$ of $\sum_{s\leq t}f(\Delta X_s)$.
As $f$ is arbitrary, it follows that  $\nu^\sigma$ is given by \eqref{eq:vt}.
\end{proof}
Let
\[
Y(t,x):=\frac{v^\sigma(t,x)}{v(x)}=\frac{1}{\sigma(X_{t-})}v\Bigl(\frac{x}{\sigma(X_{t-})}\Bigr)/v(x).
\]
Absolute continuity of $\pp^\sigma_T$ w.r.t.\ $\pp^1_T$ is guaranteed, see~\cite[Theorem~III.5.34]{JS2003}, under the condition 
\begin{equation}\label{eq:HT}
H_T=\int_0^T \int_0^\infty(\sqrt{Y(t,x)}-1)^2v(x)\,\dd x\, \dd t<\infty,\quad  \mbox{ $\pp_T^\sigma$-a.s.}
\end{equation}
Here $H_T$ has been derived from Equation (5.7) in~\cite[Chapter III]{JS2003}. 
As the driving process $L$ is a gamma process with L\'evy density $v$ given by \eqref{eq:v}, one has in fact
\(Y(t,x)=\exp\Bigl(-\beta x\Bigl(\frac{1}{\sigma(X_{t-})}-1\Bigr)\Bigr).\)
Hence, one obtains
\begin{align}
H_T & =\alpha\int_0^T \int_0^\infty \Bigl(\exp\Bigl(-\half\beta x\Bigl(\frac{1}{\sigma(X_{t-})}-1\Bigr)\Bigr)-1\Bigr)^2x^{-1}\exp(-\beta x)\,\dd x\, \dd t\nonumber\\
& =\alpha\int_0^T \int_0^\infty\left(\exp\Bigl(-\half\frac{\beta x}{\sigma(X_{t-})}\Bigr)-\exp\Bigl(-\half\beta x\Bigr)\right)^2x^{-1}\,\dd x\, \dd t 
\nonumber \\
& =: \alpha\int_0^T h_t\,\dd t.\label{eq:h}
\end{align}
Clearly, conditions on $\sigma$ have to be imposed to have absolute continuity, or even equivalence, these are given now. Of course, we still have to assume that a weak solution to \eqref{eq:sde} exists. As already announced, sufficient conditions for this will be presented in Theorem~\ref{thm:weak}.
Below, the jump measure of $X$ is denoted $\mu^X$. 
\begin{prop}\label{prop:girsanov}
Assume that $\sigma$ is a positive locally bounded measurable function on $[0,\infty)$ such that \eqref{eq:sde} admits a weak solution unique in law.  It is furthermore assumed that $\sigma$ is lower bounded by a  constant $\sigma_0>0$. Then the laws $\pp^\sigma_T$ and $\pp^1_T$ are equivalent, and one has 
\begin{eqnarray}\label{eq:lr}
\frac{\dd\pp^\sigma_T}{\dd\pp^1_T}&=&\mathcal{E}_T\left(\int_0^{\bcdot}\int_{(0,\infty)}(Y(t,x)-1)(\mu^X(\dd x,\dd t)-v(x)\,\dd x\dd t)\right),
\end{eqnarray}
where $\mathcal{E}_T$ is the Doleans exponent at time $T$ of the process within the parentheses. In other words, $Z_T:=\frac{\dd\pp^\sigma_T}{\dd\pp^1_T}$ is the solution at time $T$ to the SDE 
\begin{equation}\label{eq:sdez}
\dd Z_t= Z_{t-}\int_{(0,\infty)}(Y(t,x)-1)(\mu^X(\dd x,\dd t)-v(x)\,\dd x\,\dd t).
\end{equation}

\end{prop}
\begin{proof}
Split the integrand $h_t$ in \eqref{eq:h} into two integrals, for $x\in [0,1]$ and $x\in (1,\infty)$, call them $h_t^<$ and $h_t^>$ respectively. For $h_t^<$ we use the elementary inequality $(e^{-ax}-e^{-bx})^2\leq (b-a)^2x^2$ for $a,b,x\geq 0$ to obtain the bound 
$h_t^<\leq\frac{\beta^2}{4}(\frac{1}{\sigma(X_{t-})}-1)^2$  (here we also used $x\leq 1$), which is bounded by the finite constant $\frac{\beta^2}{2}(\frac{1}{\sigma_0^2}+1)$. To treat the integral $h_t^>$ we use the elementary inequality $(a-b)^2\leq 2(a^2+b^2)$, which leads us to study, also using $x\geq 1$,
\[
\int_1^\infty \left(\exp\Bigl(-\beta\frac{x}{\sigma(X_{t-})}\Bigr)+\exp(-\beta x)\right)\,\dd x = \exp\Bigl(-\frac{\beta}{\sigma(X_{t-})}\Bigr)\frac{\sigma(X_{t-})}{\beta}+\frac{1}{\beta}\exp(-\beta).
\]
Here the first term on the right-hand side is bounded by $\sigma(X_{t-})/\beta$.  As $X$ is increasing, $X_{t-}$ is between zero and $X_{T}$, which is finite $\pp^\sigma_T$-a.s. By the local boundedness of $\sigma$, also  $\sup_{t\leq T}\sigma(X_{t-})\leq\sup_{x\leq X_T}\sigma(x)$ is finite $\pp^\sigma_T$-a.s. From the obtained bounds on $h_t^<$ and $h_t^>$ it follows that $H_T$ is a.s.\ bounded under $\pp^\sigma_T$, so the condition \eqref{eq:HT} is satisfied. The expression for the likelihood ratio  as a Dol\'eans exponential in \eqref{eq:lr} follows from Theorem~III.5.19 in \cite{JS2003}.
\end{proof}
The next proposition gives an explicit expression for the Radon-Nikodym derivative in Proposition~\ref{prop:girsanov}. This is useful when computations with this Radon-Nikodym derivative have to be done, for instance for likelihood based inference in a statistical analysis.

\begin{prop}\label{prop:z}
Let the conditions of Proposition~\ref{prop:girsanov} hold. Then the solution $Z$ to \eqref{eq:sdez} has at any time $T>0$ 
the explicit representation 
\begin{equation}\label{eq:explicit}
Z_T=\exp\left(\int_{0}^{T}\int_{0}^{\infty}\log Y(t,x)\,\mu^X(\dd x,\dd t)-  \int_0^T\int_{0}^{\infty}(Y(t,x)-1)v(x)\,\dd x\,\dd t  \right),
\end{equation}
where both double integrals are a.s.\ finite. 
\end{prop}

\begin{proof}
It follows from Lemma~18.8 in~\cite{LS1978}, under the condition that the process $F$ defined by $F =\int_0^{\bcdot}\int_{(0,\infty)}(Y(t,x)-1)(\mu^X(\dd x,\dd t)-v(x)\,\dd x\,\dd t)$ is a process of finite variation, that the explicit expression in \eqref{eq:explicit} holds. We proceed by  
showing that $F$ has a.s.\ finite variation over any interval $[0,T]$. 
Note first that the variation of $F$ over $[0,T]$ is $\|F\|_T:=\int_{(0,T]}\int_{(0,\infty)}|Y(t,x)-1|(\mu^X(\dd x,\dd t)+v(x)\,\dd x\,\dd t)$. In view of Proposition~II.1.28 in~\cite{JS2003}, it is sufficient to check that $\int_{(0,T]}\int_{(0,\infty)}|Y(t,x)-1|v(x)\,\dd x\,\dd t$ is finite. We consider the inner integral, split into two integrals, one for $x\geq 1$, one for $0<x<1$. Consider first 
\begin{align*}
\int_{x\geq 1}|Y(t,x)-1| v(x)\,\dd x & = \int_{1}^\infty\Bigl|\exp(-\beta x(\frac{1}{\sigma(X_{t-})}-1))-1\Bigr| \frac{\alpha}{x}e^{-\beta x}\,\dd x \\
& \leq \alpha\int_1^\infty \left|\exp\Bigl(-\frac{\beta x}{\sigma(X_{t-})}\Bigr)-\exp(-\beta x)\right| \,\dd x \\
& = \frac{\alpha |\sigma(X_{t-})-1|}{\beta}\exp(-\beta).
\end{align*}
We find that $\int_{(0,T]}\int_{x\geq 1}|Y(t,x)-1|v(x)\,\dd x\,\dd t$ is finite a.s., as $\sigma$ is assumed to be a locally bounded function. For the other inner integral we use the elementary inequality for $p,q>0$
\[
\int_0^1 \frac{|e^{-px}-e^{-qx}|}{x}\,\dd x\leq |p-q|.
\]
Then we have, using that $\sigma$ is lower bounded by $\sigma_0$, 
\begin{align*}
\int_{0<x< 1}|Y(t,x)-1| v(x)\,\dd x & = \int_0^1\Bigl|\exp(-\beta x(\frac{1}{\sigma(X_{t-})}-1))-1\Bigr| \frac{\alpha}{x}e^{-\beta x}\,\dd x \\
& = \int_{0}^1\Bigl|\exp\Bigl(-\frac{\beta x}{\sigma(X_{t-})}\Bigr)-\exp(-\beta x)\Bigr| \frac{\alpha}{x}\,\dd x \\
& \leq \alpha\beta \Bigl|\frac{1}{\sigma(X_{t-})}-1\Bigr| \leq \alpha\beta \left(\frac{1}{\sigma_0}+1\right).
\end{align*}
Hence also $\int_{(0,T]}\int_{0<x< 1}|Y(t,x)-1|v(x)\,\dd x\,\dd t$ is finite.
\end{proof}


\section{Weak solutions}

We will use a variation of Proposition~\ref{prop:girsanov} to establish existence of a weak solution to $\eqref{eq:sde}$ under a growth condition on $\sigma$. The precise result follows.
\begin{thm}\label{thm:weak}
Assume that $\sigma\colon [0,\infty)\to [0,\infty)$ is measurable, lower bounded by a constant $\sigma_0>0$, and satisfies a linear growth condition, i.e.\ there exists $K>0$ such that for all $x\geq 0$ it holds that $\sigma(x)\leq K(1+x)$. Then,  on the interval $[0,\infty)$, Equation~\eqref{eq:sde} admits a weak solution that is unique in law.
\end{thm}
\begin{proof}
This proof is inspired by Section~5.3B of \cite{karatzas91} for a similar problem in a Brownian setting. Fix $T>0$ and consider a probability space $(\Omega,\cf,\qq)$ on which $X$ is defined as the gamma process. We choose $\Omega$ to be the Skorohod space, $\cf=\cf^X=\sigma(X_t,t\geq 0)$, and $X$ the coordinate process. Furthermore we use the filtration $\mathbb{F}^X=\{\cf^X_t, t\geq 0\}$. The restriction of $\qq$ to $\cf^X_T$ is denoted $\qq_T$. As a semimartingale, under $\qq$, $X$ has third characteristic $\nu^{X,\qq}(\dd x,\dd t)=v(x)\dd x\dd t$ with $v$ as in \eqref{eq:v}. Define $L$ by $\dd L_t=\frac{1}{\sigma(X_{t-})}\,\dd X_t$ and $L_0=0$. Since $\sigma$ is bounded from below and measurable, the process $L$ is well-defined.
We again take 
\[
Y(t,z)=\frac{1}{\sigma(X_{t-})}\frac{v(\frac{z}{\sigma(X_{t-})})}{v(z)},
\]
and make a measure change on $\cf^X_T$, parallel to Proposition~\ref{prop:girsanov},
\[
\frac{\dd \pp_T}{\dd\qq_T}=Z_T:=\mathcal{E}_T\left(\int_0^{\bcdot}\int_{(0,\infty)}(Y(t,z)-1)(\mu^X(\dd z,\dd t)-\nu^{X,\qq}(\dd z,\dd t))\right).
\]
Provided that $\pp_T$ is a probability measure on $\cf^X_T$, which happens if $\ee_\qq Z_T=1$,
the third characteristic of $X$ under $\pp_T$ is, similar to Lemma~\ref{lemma:vsigma},
\[
\nu^{X,\pp}(\dd z,\dd t)=Y(t,z)\nu^{X,\qq}(\dd z,\dd t)=v^{X,\pp}(t,z)\dd z\dd t,
\]
where $v^{X,\pp}(t,z)=\frac{1}{\sigma(X_{t-})}v(\frac{z}{\sigma(X_{t-})})$, so $v^{X,\pp}(t,z)=v^\sigma$ with $v^\sigma$ as in \eqref{eq:vt}.
By the arguments in the proof of Lemma~\ref{lemma:vsigma}, one obtains that under $\pp_T$ the process $L$ has third characteristic $\nu^{L,\pp}(\dd z,\dd t)=v^{L,\pp}(t,z)\dd z\dd t$,
with $v^{L,\pp}(t,z)=v^{X,\pp}(t,z\sigma(X_{t-}))\sigma(X_{t-})$, which is nothing else but $v(z)$, implying that under $\pp_T$, $L$ is a gamma process on $[0,T]$, and it also holds that $\dd X=\sigma(X_{t-})\,\dd L_t$. We conclude that under $\pp_T$, $X$ is
a solution of the SDE, where $L$ is a gamma process with L\'evy density $v$. 
What remains to be shown for existence of a weak solution is that $\pp_T$ is a probability measure on $\cf^X_T$. 
We use Theorem~IV.3 of \cite{LepingleMemin}, 
this will be done via a detour as a direct application doesn't give the desired results.
First we compute $\int_0^\infty (y(\sigma(x),z)\log y(\sigma(x),z)-y(\sigma(x),z)+1)v(z) \,\dd z$, where $y(\sigma,z)=\frac{1}{\sigma}v(\frac{z}{\sigma})/v(z)=\exp(-\beta z/\sigma +\beta z)$. 

Consider
$f(\sigma) 
=\int_0^\infty (y(\sigma,z)\log y(\sigma,z)-y(\sigma,z)+1)v(z) \,\dd z$. Some tedious calculations show that
$f(\sigma)=\alpha(\sigma-1-\log \sigma)$. 
Hence
\[
\int_{(0,\infty)}(Y(t,z)\log Y(t,z)-Y(t,z)+1)v(z)\dd z=\alpha(\sigma(X_{t-})-1-\log \sigma(X_{t-})).
\]
Consider for some $\delta>0$, to be chosen later, $T_n=n\delta T$, for $n=0,\ldots N$, with $T_N=T$, so $\delta=1/N$; note that $T_n-T_{n-1}=\delta T$. Let $Z^n$ be the solution to 
\begin{equation}\label{eq:zn}
\dd Z^n_t=Z^n_{t-}\left(\int_{(0,\infty)}\one_{(T_{n-1},T_n]}(t)(Y(t,x)-1)(\mu^X(\dd x,\dd t)-\nu^{X,\qq}(\dd x,\dd t))\right), 
\end{equation}
with $Z^n_0=1$, for $n=1,\ldots,N$. If the $Z^n$ are martingales, not just local martingales, under $\qq$ w.r.t.\ the filtration $\mathbb{F}^X$, then $\ee_\qq[Z^n_{T_n}|\cf^X_{T_{n-1}}]=Z^n_{T_{n-1}}=1$. Since $Z_T=\prod_{n=1}^N Z^n_{T_n}$, one obtains
\[
\ee_\qq Z_T=\ee_\qq\prod_{n=1}^{N-1}Z^n_{T_n}\ee_\qq[Z^N_{T_N}|\cf_{T_{N-1}}]=\ee_\qq\prod_{n=1}^{N-1}Z^n_{T_n}=\ee_\qq Z_{T_{N-1}},
\]
which can be seen equal to one by an induction argument. To see that the $Z^n$ are martingales, we use Theorem~IV.3 of \cite{LepingleMemin}, i.e.\ the aim is to show
\[
\ee_\qq\exp\left(\int_0^T\int_{[0,\infty)}\one_{(T_{n-1},T_n]}(t)(Y(t,z)\log Y(t,z)-Y(t,z)+1)v(z)\dd z\dd t\right)<\infty.
\]
In view of the computations above this amounts to showing that 
\[
\ee_\qq\exp\left(\int_{T_{n-1}}^{T_n}\alpha(\sigma(X_{t-})-1-\log \sigma(X_{t-}))\dd t\right)<\infty.
\]
Using that $\sigma$ is lower bounded and satisfies the growth condition, and that $X$ is an increasing process under $\qq$, we have that the integrand above is upper bounded by $\alpha(K(1+X_t)-1-\log\sigma_0\leq C+\alpha KX_T$, for some $C>0$. Hence, $\int_{T_{n-1}}^{T_n}\alpha(\sigma(X_{t-})-1-\log \sigma(X_{t-}))\dd t\leq\delta\alpha K T X_T+C'$, where $C'$ is another positive constant, so it is sufficient to prove that $\ee_\qq\exp(\delta\alpha K T X_T)<\infty$. But, under $\qq$, $X_T$ has a $\operatorname{Gamma}(\alpha T,\beta)$ distribution, so the latter expectation, calculated as an integral w.r.t.\ to the gamma distribution,  is seen to be finite if $\delta<\beta/(\alpha K T)$, equivalently $N>\alpha K T/\beta$. For such a choice of $N$, we obtain that the $Z^n$ are  martingales and hence $\ee_\qq Z_T=1$. As a consequence $\pp_T$ is a probability measure on $(\Omega,\cf^X_T)$ for every $T>0$. In fact the $\pp_T$ form a consistent family of probability measures, and hence there exists a probability measure $\pp$ on $\cf$ such that $\pp|_{\cf^X_T}=\pp_T$, see Lemma~18.18 in \cite{Kallenberg}. This shows that there exists a weak solution on the entire time interval $[0,\infty)$. Finally, we turn to uniqueness in law of a weak solution. Consider two possible weak solutions $X^i$, or rather $(X^i,L^i)$ for $i=1,2$, on an interval $[0,T]$, defined on their own filtered probability spaces $(\Omega^i,\cf^i,\mathbb{F}^i,\pp^i)$. Consider changes of measures $\dd\tilde\pp^i=Z^i_T\dd\pp^i$. Here we take $T=t$ in
\[
Z^i_t=\mathcal{E}_t\left(\int_0^{\bcdot}\int_{(0,\infty)}(\tilde Y^i(s,z)-1)(\mu^{X^i}(\dd z,\dd s)-\nu^{X^i,\pp^i}(\dd z,\dd s))\right),
\]
where 
\[
\tilde Y^i(s,z)=\frac{\sigma(X^i_{s-})v(z)}{v(z/\sigma(X^i_{s-}))}=\tilde y(\sigma(X^i_{s-}),z),
\]
with $\tilde y(\sigma,z)=\frac{1}{y(\sigma,z)}=\exp(\beta z/\sigma-\beta z)$. We assume for a while that the $Z^i_T$ have expectations one under $\pp^i$ so that the $\tilde\pp^i$ are probability measures on $\cf^i_T$, equivalent to the $\pp^i$. By the arguments used earlier in this proof, under $\tilde\pp^i$ the processes $X^i$ are gamma processes with parameters $\alpha,\beta$. Hence the distributions, for $i=1,2$, of the $X^i$ are identical under the probability measures $\tilde\pp^i$. Consider then samples $X^i(n)=(X^i_{t_1},\ldots X^i_{t_n})$, $0\leq t_1\leq,\ldots,\leq t_n\leq T$, and Borel sets $B$ of $\rr^n$. Then
\begin{equation}\label{eq:ppi}
\pp^i(X^i(n)\in B)=\ee_{\tilde\pp^i}\frac{1}{Z^i_T}\one_{\{X^i(n)\in B\}}.
\end{equation}
On the right hand side of Equation~\eqref{eq:ppi}, all random quantities are defined in terms of $X^i$, hence the expectations in \eqref{eq:ppi} are the same for $i=1,2$. This shows that the finite dimensional distributions of $X^1$ and $X^2$ are identical and hence the laws of $X^1$ and $X^2$ are  the same as well. It is left to show that the $Z^i_T$ have expectations one under $\pp^i$. We follow the same path as above, we first compute $\tilde{F}(x):=\int (\tilde y(\sigma(x),z)\log \tilde y(\sigma(x),z)-\tilde y(\sigma(x),z)+1)\frac{v(z/\sigma)}{\sigma} \,\dd z$. Thereto we consider $\tilde f(\sigma) 
=\int (\tilde y(\sigma,z)\log \tilde y(\sigma,z)-\tilde y(\sigma,z)+1)\frac{v(z/\sigma)}{\sigma} \,\dd z$. It turns out that $\tilde f(\sigma)=\alpha(\frac{1}{\sigma}-1+\log \sigma)$, so $\tilde f(\sigma)=f(\frac{1}{\sigma})$, and hence $\tilde F(x)=\alpha(\frac{1}{\sigma(x)}-1+\log \sigma(x))$. From here we continue to solve the SDE for $Z^i$ on intervals $(T_{n-1},T_n]$, similar to \eqref{eq:zn}, with $T_n=n\delta T$, resulting in  processes $Z^{i,n}$ that are martingales with $Z^{i,n}_{T_{n-1}}=1$. To show the martingale property for $Z^{i,n}$, we use again Theorem~IV.3 of \cite{LepingleMemin}, i.e.\ we show
\[
\ee_{\pp^i}\exp\left(\int_{T_{n-1}}^{T_n}\alpha\Bigl(\frac{1}{\sigma(X^i_{t-})}-1+\log \sigma(X^i_{t-})\Bigr)\dd t\right)<\infty.
\]
Here the integrand is bounded by $\alpha(\frac{1}{\sigma_0}-1 +\log K +\log(1+X^i_T))$. Hence for a constant $C$, depending on $T$, the exponent is less than or equal to $C(1+X^i_T)^{\delta\alpha T}$, and we have to show for a well chosen $\delta>0$ that $\ee_{\pp^i}(1+X^i_T)^{\delta\alpha T}<\infty$, equivalently $\ee_{\pp^i}(X^i_T)^{\delta\alpha T}<\infty$. The conditions in Proposition~4.1 of \cite{KlebanerLiptser} (on their operator $L_s(x_{s-})$) are satisfied by the linear growth condition on $\sigma$, and as a result one has $\ee_{\pp^i}(X^i_T)^{2}<\infty$. Therefore we take $\delta<2/\alpha T$.
\end{proof}


\bibliographystyle{apa-good}
\bibliography{bibliography}

\end{document}